\documentclass[11pt,a4paper]{amsart}
\usepackage{hyperref}
\usepackage{amsfonts}
\usepackage{amsthm}
\usepackage{amsmath}
\usepackage{amscd}
\usepackage[latin2]{inputenc}
\usepackage{t1enc}
\usepackage[mathscr]{eucal}
\usepackage{indentfirst}
\usepackage{graphicx}
\usepackage{graphics}
\usepackage{pict2e}
\usepackage{epic}
\numberwithin{equation}{section}
\usepackage[margin=2.9cm]{geometry}
\usepackage{epstopdf}

\theoremstyle{plain}
\newtheorem{thm}{Theorem}[section]
\newtheorem{lem}[thm]{Lemma}
\newtheorem{cor}[thm]{Corollary}
\newtheorem{prop}[thm]{Proposition}

\theoremstyle{definition}

\newtheorem{?}[thm]{Problem}

\theoremstyle{definition}
\newtheorem*{nt*}{Notation}

\newcommand{\cala}{\mathcal A}
\newcommand{\calc}{\mathcal C}
\newcommand{\calh}{\mathcal H}
\newcommand{\calj}{\mathcal J}

\newcommand {\C} {\mathbb C}
\newcommand {\R} {\mathbb R}
\newcommand {\Z} {\mathbb Z}
\newcommand {\D} {\mathbb D}
\newcommand {\T} {\mathbb T}
\newcommand {\mba} {\mathbb A}
\newcommand {\mbb} {\mathbb B}
\newcommand {\mbh} {\mathbb H}
\newcommand {\mbo} {\mathbb O}

\newcommand {\bfa} {\mathbf a}
\newcommand {\bfb} {\mathbf b}
\newcommand {\bfc} {\mathbf c}

\newcommand {\bfp} {\mathbf p}

\newcommand {\bfq} {\mathbf q}
\newcommand {\bfr} {\mathbf r}

\newcommand {\spc} {\mathcal{B}_\ell(\mathbb{D}^d)}
\newcommand {\dom} {\mathbb{D}^d}
\newcommand {\pa} {\partial}
\newcommand {\zd} {\mathbb{Z}_{1,d}}
\newcommand {\tha} {\theta}
\newcommand {\ome} {\omega}
\newcommand {\al} {\alpha}

\begin{document}

\title[]{real symmetric, unitary, and complex symmetric weighted composition operators on Bergman spaces of polydisk}

\author{Pham Viet Hai}%
\address[P. V. Hai]{School of Applied Mathematics and Informatics, Hanoi University of Science and Technology, Vien Toan ung dung va Tin hoc, Dai hoc Bach khoa Hanoi, 1 Dai Co Viet, Hanoi, Vietnam.}%
\email{hai.phamviet@hust.edu.vn}



\begin{abstract}
In this paper, we study weighted composition operators on Bergman spaces of analytic functions which are square integrable on polydisk. We develop the study in full generality, meaning that the corresponding weighted composition operators are not assumed to be bounded. The properties of weighted composition operators such as real symmetry, unitariness, complex symmetry, are characterized fully in simple algebraic terms, involving their symbols. As it turns out, a weighted composition operator having a symmetric structure must be bounded. We also obtain the interesting consequence that real symmetric weighted composition operators are complex symmetric corresponding an adapted and highly relevant conjugation.
\end{abstract}

\maketitle

\section{Introduction}
\subsection{Weighted composition operators}
Our primary object of study is weighted composition operators that are defined as follows. Let $\mbo$ be a Banach space of analytic functions on some domain $V\subseteq\C^d$, where $d\in\Z_{\geq 1}$. Given functions $g:V\to V,f:V\to\C$ (often called the \emph{symbols}), we consider the formal expression $\widetilde{[f,g]}$ by setting
\begin{gather*}
    (\widetilde{[f,g]}\xi)(z)=f(z)\xi(g(z)),\quad\text{where $\xi(\cdot)\in\mbo,z\in V$}.
\end{gather*}
The \emph{maximal weighted composition operator} is defined as
\begin{gather*}
    \text{dom}(W_{f,g,\max})=\{\xi(\cdot)\in\mbo:\widetilde{[f,g]}\xi(\cdot)\in\mbo\},\\
    W_{f,g,\max}\xi(\cdot)=\widetilde{[f,g]}\xi(\cdot),\quad\xi(\cdot)\in\text{dom}(W_{f,g\max}).
\end{gather*}
The operator $W_{f,g}$ is called a \emph{non-maximal weighted composition operator} if the inclusion $W_{f,g}\preceq W_{f,g,\max}$ holds. By the term \emph{bounded weighted composition operator}, as it is used in the paper, is meant an operator $W_{f,g}$ which satisfies (i) $\text{dom}(W_{f,g})=\mbo$ and (ii) there exists a constant $R>0$ such that
\begin{gather*}
    \|W_{f,g}h\|\leq R\|h\|\quad\forall h(\cdot)\in\mbo.
\end{gather*}
Note that the phrase "unbounded" is understood as "not necessarily bounded"; in other words, bounded operators belong to the unbounded class. When $f(\cdot)\equiv 1$, writing $C_g$ instead of $W_{1,g}$ and it is named a \emph{composition operator}.

The earliest reference on weighted composition operators appears to be the Banach-Stone theorem, which states that the only surjective isometries between Banach spaces of real-valued continuous functions are precisely of this type (\cite{zbMATH00467266}). Leeuw et al \cite{zbMATH03158247} and Forelli \cite{zbMATH03213975} obtained the similar results for Hardy spaces. Since then, weighted composition operators have become the subject matter of intensive and extensive research on various spaces and they have made important connections to the study of other operators (\cite{zbMATH03607039, zbMATH00933265}). What makes weighted composition operators worth to study is the fact that their properties can be characterized in simple algebraic terms. 

When $\mbo$ is Hardy space of the unit disk, there is a large body of work on composition operators (see \cite{zbMATH00918588, zbMATH00473375}). It was proven long by Littlewood \cite{zbMATH02591644} that composition operators are always bounded. In contrast to the unweighted setting, the boundedness of weighted composition operators has not been well understood. There are many examples (see \cite{zbMATH00918588}) which show that weighted composition operators are not bounded. Other properties have been studied such as conditions for (weighted) composition operators to be real symmetric \cite{zbMATH05815825}, be normal \cite{zbMATH05080402, zbMATH06074411}, be unitary \cite{zbMATH06074411}, or be invertible \cite{zbMATH05907565}.

Considered on Fock spaces, weighted composition operators act very differently; for example, not all composition operators are bounded. Carswell, MacCluer and Schuster \cite{zbMATH02099157} showed that the only affine function $g(z)=\al z+\beta$ gives rise to the boundedness. Le \cite{zbMATH06324457} gave a complete characterization of bounded weighted composition operators on Fock spaces. Characterizations of compactness, isometry, normality, cohyponormality, and supercyclicity have also been produced, including those of Le \cite{zbMATH06324457}, Hai and Khoi \cite{zbMATH06627851}, and Carroll and Gilmore \cite{MR4248470}.

\subsection{Complex symmetric operators}
A \emph{complex symmetric operator} is an unbounded linear operator $X:\text{dom}(X)\subseteq\calh\to\calh$ on a Hilbert space $\calh$ with the property that $X=\calc X^*\calc$, where $\calc$ is an isometric involution (in short, \emph{conjugation}) on $\calh$. To indicate the dependence on $\calc$, this case is often called \emph{$\calc$-selfadjoint}.

The concept of complex symmetric operators is a natural generalization of complex symmetric matrices, and their general research was commenced by Garcia, Putinar \cite{zbMATH02237890, zbMATH05148120}. The class of complex symmetric operators is of great interest for several reasons. First, it has deep classical roots in various branches of mathematics, including matrix theory, function theory, projective geometry, etc. See \cite{zbMATH02237890, zbMATH05148120, zbMATH06349831} for details. Second, it is large enough to cover well-known operators such as normal operators, Volterra operators, Hankel operators, compressed Toeplitz operators. More interestingly, differential operators relevant in non-hermitian quantum mechanics obeying a parity and time symmetry, so called \emph{$\mathcal{PT}$-symmetric}, belong to this class (see \cite{zbMATH01365858}).

Putted forward by Garcia and Hammond \cite{zbMATH06454968}, the problem of classifying complex symmetric weighted composition operators is of interest and it is being answered in several particular contexts. Jung et al \cite{zbMATH06320823} considered bounded weighted composition operators on Hardy spaces and characterized those which are complex symmetric with respect to the conjugation $\calj f(z)=\overline{f(\overline{z})}$. The works \cite{zbMATH06454968, zbMATH06320823} make a motivation to study the problem in various spaces. We refer the reader to \cite{zbMATH06487337, zbMATH07216720} for Fock space and \cite{zbMATH07053055} for Hardy spaces in several variables and \cite{zbMATH06882577} for Lebesgue spaces.


\subsection{Aim}
In this paper, we study weighted composition operators on Bergman spaces of analytic functions which are square integrable on polydisk. We develop the study in full generality, meaning that the corresponding weighted composition operators are not assumed to be bounded. The properties of weighted composition operators such as real symmetry, unitariness, complex symmetry, are characterized fully in simple algebraic terms, involving their symbols. As it turns out, a weighted composition operator having a symmetric structure must be bounded. We also obtain the interesting consequence that real symmetric weighted composition operators are complex symmetric corresponding an adapted and highly relevant conjugation. 

\section{Preparation}\label{sec2}
\subsection{Notations}
Before getting closer to the content, we list notations and terminologies. The domain of an operator is denoted as $\text{dom}(\cdot)$. When dealing with unbounded operators, the symbol $A\preceq B$ means that $\text{dom}(A)\subseteq\text{dom}(B)$ and $Ax=Bx$ for $x\in\text{dom}(A)$. Let $\zd=\{1,2,\cdots,d\}$. Always denote
\begin{gather}\label{eq-omep}
    \ome_p(z)=\dfrac{\overline{p}}{p}\cdot\dfrac{p-z}{1-\overline{p}z},\quad p\in\D\setminus\{0\}.
\end{gather}
Let $\D=\{z\in\C:|z|<1\}$ and $\T=\{z\in\C:|z|=1\}$. For a fixed integer $d$, the polydisk $\dom$ of $\C^d$ is the Cartesian product of $d$ copies of $\D$. Bergman space $\spc$ consists of all analytic functions $\xi:\dom\to\C$ with
\begin{gather*}
    \|\xi\|^2=\int\limits_{\dom}|\xi(z)|^2\prod_{j=1}^d(1+\ell_j)(1-|z_j|^2)^{\ell_j}\,dA(z_j),
\end{gather*}
where $dA$ denotes the normalized area measure on $\D$. The reproducing kernel in $\spc$ is given by
\begin{gather*}
    K_z(u)=\prod_{j=1}^d\dfrac{1}{(1-u_j\overline{z_j})^{\ell_j+2}}\quad\forall z,u\in\dom.
\end{gather*}

\subsection{Elementary estimate}
The following lemma may be well-known; but the proof is given, for a completeness of exposition.
\begin{lem}\label{lem-202106091438}
Let $\bfc\in\D^t$ and $s\in\Z_{\geq 1}$. Then there exists a constant $\Delta$ such that
\begin{gather*}
    \int\limits_{\D^s}|h(\bfc,y)|^2\prod_{j=1}^s(1-|y_j|^2)^{\ell_j}dA(y_j)
    \leq\Delta\int\limits_{\D^t\times\D^s}|h(z)|^2\prod_{j=1}^{t+s}(1-|z_j|^2)^{\ell_j}dA(z_j).
\end{gather*}
Consequently, if $h(\cdot)\in\mathcal{B}_\ell(\D^t\times\D^s)$, then the function $h_{\bfc}(\cdot)\in\mathcal{B}_\ell(\D^s)$, where $h_{\bfc}(y)=h(\bfc,y)$.
\end{lem}
\begin{proof}
Denote $\delta_j=(1-|\bfc_j|)/2$ and $D(\bfc,\delta)^t=\{z=(z_1,z_2,\cdots,z_t)\in\C^t:|z_j-c_j|<\delta_j\,\forall 1\leq j\leq t\}$. Setting
\begin{gather*}
    W=\int\limits_{\D^t\times\D^s}|h(z)|^2\prod_{j=1}^{t+s}(1-|z_j|^2)^{\ell_j}dA(z_j),
\end{gather*}
by the Fubini theorem, we have
\begin{gather*}
    W=
    \int\limits_{\D^s}\prod_{j=1}^s(1-|y_j|^2)^{\ell_j}dA(y_j)
    \int\limits_{\D^t}|h(x,y)|^2\prod_{j=1}^t(1-|x_j|^2)^{\ell_j}dA(x_j)\\
    \geq\int\limits_{\D^s}\prod_{j=1}^s(1-|y_j|^2)^{\ell_j}dA(y_j)
    \int\limits_{D(\bfc,\delta)^t}|h(x,y)|^2\prod_{j=1}^t(1-|x_j|^2)^{\ell_j}dA(x_j)\\
    \geq
    \prod_{j=1}^t\left(1-\left(\dfrac{1+|\bfc_j|}{2}\right)^2\right)^{\ell_j}
    \int\limits_{\D^s}\prod_{j=1}^s(1-|y_j|^2)^{\ell_j}dA(y_j)
    \int\limits_{D(\bfc,\delta)^t}|h(x,y)|^2dA(x_j),
\end{gather*}
which implies, by the subharmonic property of the function $x\mapsto|h(x,y)|^2$, that
\begin{gather*}
    W
    \geq\prod_{j=1}^t\left(1-\left(\dfrac{1+|\bfc_j|}{2}\right)^2\right)^{\ell_j}
    \int\limits_{\D^s}\prod_{j=1}^s(1-|y_j|^2)^{\ell_j}dA(y_j)\pi^t\prod_{j=1}^t\delta_j|h(\bfc,y)|^2\\
    =\pi^t\prod_{j=1}^t\delta_j\left(1-\left(\dfrac{1+|\bfc_j|}{2}\right)^2\right)^{\ell_j}
    \int\limits_{\D^s}|h(\bfc,y)|^2\prod_{j=1}^s(1-|y_j|^2)^{\ell_j}dA(y_j).
\end{gather*}
\end{proof}

We recall a condition when a linear fractional function is a self-mapping of $\D$.
\begin{lem}[{\cite{zbMATH06077514}}]\label{202202242145}
Let
\begin{gather*}
    \xi(z)=\dfrac{az+b}{cz+d},\quad\text{where $a,b,c,d\in\C$ with $ad-bc\ne 0$}.
\end{gather*}
Then $\xi(\cdot)$ is a self-mapping of $\D$ if and only if
\begin{gather*}
    |b\overline{d}-a\overline{c}|+|ad-bc|\leq|d|^2-|c|^2.
\end{gather*}
\end{lem}

\subsection{Algebraic observations}
\begin{lem}\label{20220224}
Let $f:\dom\to\C,g:\D\to\C$ be analytic functions, where $d\in\Z_{\geq 1}$. Suppose that for every $\kappa\in\zd$, the product $g(u_\kappa)f(u_1,u_2,\cdots,u_n)$ is a function of variables $u_1,u_2,\cdots,u_{\kappa-1},u_{\kappa+1},\cdots,u_n$. Then for every $\kappa\in\zd$ the product $f(u_1,u_2,\cdots,u_n)\prod_{j=1}^{\kappa}g(u_j)$ is a function of variables $u_{\kappa+1},u_{\kappa+2},\cdots,u_n$.
\end{lem}
\begin{proof}
The lemma is proven by induction on $\kappa$ and its proof is left to the reader. 
\end{proof}

\begin{lem}\label{lem-202103262022}
Let $\psi:\D\to\D$ be an analytic function with the property that
\begin{gather}\label{eq-202103262022}
    \psi'(u)(1-u\overline{\psi(z)})^2
    =\overline{\psi'(z)}(1-\psi(u)\overline{z})^2\quad\forall z,u\in\D.
\end{gather}
Then
\begin{gather}\label{202208211011}
    \psi(z)=\bfa+\dfrac{\bfb z}{1-\overline{\bfa}z},
\end{gather}
where coefficients satisfy
\begin{gather}
    \label{202208211012}\bfa\in\D,\bfb\in\R,\\
    \label{202208211013}|\bfa(\bfb-|\bfa|^2+1)|+|\bfb|\leq 1-|\bfa|^2.
\end{gather}
\end{lem}
\begin{proof}
It is clear that the case when $\psi(\cdot)\equiv\text{const}$ verifies \eqref{eq-202103262022}. Next is to consider the remaining case $\psi(\cdot)\not\equiv \text{const}$. Let $z_\star\in\D$ such that $\psi(z_\star)\ne 0$. Letting $z=z_\star$ in \eqref{eq-202103262022}, we can observe
\begin{gather*}
    \dfrac{\psi'(u)}{(1-\psi(u)\overline{z_\star})^2}
    =\dfrac{\overline{\psi'(z_\star)}}{(1-u\overline{\psi(z_\star)})^2}.
\end{gather*}
After integrating with respect to the variable $u$, $\psi(\cdot)$ is of form
\begin{gather*}
    \psi(z)=\dfrac{Az+B}{Cz+D},
\end{gather*}
where $A,B,C,D$ are complex constants. There are two cases of the coefficient $C$. If $C=0$, then
\begin{gather*}
    \psi(z)=\widehat{A}z+\widehat{B},\quad\text{where $\widehat{A}=\dfrac{A}{D}$ and $\widehat{B}=\dfrac{B}{D}$}.
\end{gather*}
Through setting $\psi(z)=\widehat{A}z+\widehat{B}$ in \eqref{eq-202103262022} and then equating coefficients, we get
\begin{gather*}
    \widehat{A}\in\R,\quad\widehat{B}=0.
\end{gather*}
If $C\ne 0$, then
\begin{gather*}
    \psi(z)=G+\dfrac{E}{z+F},\quad\text{where $G=\dfrac{A}{C},E=\dfrac{BC-AD}{C^2},F=\dfrac{D}{C}$}.
\end{gather*}
It follows from $\psi(\cdot)\not\equiv\text{const}$, that $E\ne 0$. By way of substituting this form of $\psi(\cdot)$ back into \eqref{eq-202103262022} and then equating coefficients, we obtain
\begin{gather*}
    \dfrac{\overline{G}^2}{G^2}=\dfrac{\overline{E}}{E},
    \quad\dfrac{\overline{G}}{G}=\dfrac{\overline{F}}{F},
    \quad GF+E=-\dfrac{E\overline{G}}{\overline{E}G}.
\end{gather*}
Thus, \eqref{202208211011}-\eqref{202208211012} hold, where
\begin{gather*}
    \bfa=G+\dfrac{E}{F},\quad\bfb=-\dfrac{E}{F^2};
\end{gather*}
meanwhile, \eqref{202208211013} follows directly from Lemma \ref{lem-202106091438}.
\end{proof}

\begin{lem}\label{lem-202105202050}
Let $\tha\in\T$ and $\psi:\D\to\D$ be an analytic function with the property that
\begin{gather}\label{eq-202105202050}
    \psi'(y)(1-\tha y\psi(x))^2
    =\psi'(x)(1-\tha\psi(y)x)^2\quad\forall x,y\in\D.
\end{gather}
Then
\begin{gather*}
    \psi(z)=\al_0+\dfrac{\al_1\tha z}{1-\al_0\tha z},
\end{gather*}
where coefficients satisfy
\begin{gather*}
    |\al_0+\overline{\al_0}(\al_1-\al_0^2)|+|\al_1|\leq 1-|\al_0|^2.
\end{gather*}
\end{lem}
\begin{proof}
It is clear that the case when $\psi(\cdot)\equiv\text{const}$ verifies \eqref{eq-202105202050}. Next is to consider the remaining case $\psi(\cdot)\not\equiv \text{const}$. Let $z_\star\in\D$ such that $\psi(z_\star)\ne 0$. Letting $x=z_\star$ in \eqref{eq-202105202050}, we get
\begin{gather*}
    \dfrac{\psi'(y)}{(1-\tha\psi(y)z_\star)^2}
    =\dfrac{\psi'(z_\star)}{(1-\tha y\psi(z_\star))^2}.
\end{gather*}
By way of integrating with respect to the variable $y$, we observe $\psi(\cdot)$ is of form
\begin{gather*}
    \psi(z)=\dfrac{Az+B}{Cz+D},
\end{gather*}
where $A,B,C,D$ are complex constants. There are two cases of $C$. If $C=0$, then
\begin{gather*}
    \psi(z)=\widehat{A}z+\widehat{B},\quad\text{where $\widehat{A}=\dfrac{A}{D}$ and $\widehat{B}=\dfrac{B}{D}$}.
\end{gather*}
Through setting $\psi(z)=\widehat{A}z+\widehat{B}$ in \eqref{eq-202105202050} and then equating coefficients, we get $\widehat{B}=0$. Consider when $C\ne 0$ and then
\begin{gather*}
    \psi(z)=G+\dfrac{E}{z+F},\quad\text{where $G=\dfrac{A}{C},E=\dfrac{BC-AD}{C^2},F=\dfrac{D}{C}$}.
\end{gather*}
It follows from $\psi(\cdot)\not\equiv\text{const}$, that $E\ne 0$. By way of substituting this form of $\psi(\cdot)$ back into \eqref{eq-202105202050} and then equating coefficients, we obtain
\begin{gather*}
    \begin{cases}
    \text{either $G=F=0,E=\overline{\tha}$},\\
    \text{or $E+GF=-\overline{\tha}$}.
    \end{cases}
\end{gather*}
Thus, the conclusion holds, where
\begin{gather*}
    \al_0=G+\dfrac{E}{F},\quad\al_1=-\dfrac{E}{F^2}\overline{\tha}.
\end{gather*}
\end{proof}

\begin{lem}\label{lem-202105291538}
Let $p\in\D\setminus\{0\}$ and $\ome_p(\cdot)$ be the function given by \eqref{eq-omep}. If the function $\psi:\D\to\D$ satisfies
\begin{gather}\label{eq-202105270751}
    \ome_p'(y)\psi'(u)[1-u\ome_p(\psi(y))]^2
    =\ome_p'(\psi(y))\psi'(y)[1-\ome_p(y)\psi(u)]^2,
\end{gather}
then either
\begin{gather}\label{202208121547}
    \psi\equiv\text{const}
\end{gather}
or
\begin{gather}\label{202208121548}
    \psi(z)=G+\dfrac{E}{z+F},
\end{gather}
where coefficients satisfy
\begin{gather}
    \label{202208121549}p-G|p|^2=-F|p|^2+(GF+E)\overline{p},\\
    \label{202208121550}E\ne 0,\\
    \label{202208121551}|(GF+E)\overline{F}-G|+|E|\leq|F|^2-1.
\end{gather}
\end{lem}
\begin{proof}
It is clear that the case when $\psi(\cdot)\equiv\text{const}$ verifies \eqref{eq-202105270751}. Next is to consider the remaining case $\psi(\cdot)\not\equiv \text{const}$. Since $\ome_p'(z)=\frac{\overline{p}}{p}\cdot\frac{|p|^2-1}{(1-\overline{p}z)^2}$, equation \eqref{eq-202105270751} is reduced to the following
\begin{gather}\label{eq-202105270808}
    \dfrac{\psi'(u)[1-u\ome_p(\psi(y))]^2}{(1-\overline{p}y)^2}
    =\dfrac{\psi'(y)[1-\ome_p(y)\psi(u)]^2}{(1-\overline{p}\psi(y))^2}
\end{gather}
or equivalently to saying that
\begin{gather*}
    \dfrac{\psi'(u)}{[1-\ome_p(y)\psi(u)]^2}
    =\dfrac{\psi'(y)}{(1-\overline{p}\psi(y))^2}\cdot\dfrac{(1-\overline{p}y)^2}{[1-u\ome_p(\psi(y))]^2}.
\end{gather*}
Like the arguments in Lemmas \ref{lem-202105202050} and \ref{lem-202103262022}, $\psi(\cdot)$ is of form
\begin{gather}\label{eq-202105270818}
    \psi(z)=\dfrac{Az+B}{Cz+D},
\end{gather}
where $A,B,C,D$ are complex constants. Setting \eqref{eq-omep} in \eqref{eq-202105270808}, we find
\begin{gather*}
    \psi'(u)[p(1-\overline{p}\psi(y))-\overline{p}u(p-\psi(y))]^2
    =\psi'(y)[p(1-\overline{p}y)-\overline{p}(p-y)\psi(u)]^2,
\end{gather*}
which implies, by \eqref{eq-202105270818}, that
\begin{gather*}
    [(\overline{p}A-|p|^2C)uy+(\overline{p}B-|p|^2D)u+(pC-|p|^2A)y+pD-|p|^2B]^2\\
    =[(\overline{p}A-|p|^2C)uy+(pC-|p|^2A)u+(\overline{p}B-|p|^2D)y+pD-|p|^2B]^2.
\end{gather*}
There are two cases.

- If
\begin{gather*}
    (\overline{p}A-|p|^2C)uy+(\overline{p}B-|p|^2D)u+(pC-|p|^2A)y+pD-|p|^2B\\
    =(\overline{p}A-|p|^2C)uy+(pC-|p|^2A)u+(\overline{p}B-|p|^2D)y+pD-|p|^2B,
\end{gather*}
then after equating coefficients, we get $\overline{p}B-|p|^2D=pC-|p|^2A$. It then implies \eqref{202208121549}, where
\begin{gather*}
    G=\dfrac{A}{C},\quad F=\dfrac{D}{C},\quad E=\dfrac{BC-AD}{C^2}.
\end{gather*}
Note that condition \eqref{202208121551} follows directly from Lemma \ref{202202242145}.

- If
\begin{gather*}
    (\overline{p}A-|p|^2C)uy+(\overline{p}B-|p|^2D)u+(pC-|p|^2A)y+pD-|p|^2B\\
    =-(\overline{p}A-|p|^2C)uy-(pC-|p|^2A)u-(\overline{p}B-|p|^2D)y-pD+|p|^2B,
\end{gather*}
then after equating coefficients, we get
\begin{gather*}
    \begin{cases}
    \overline{p}A-|p|^2C=0,\\
    \overline{p}B-|p|^2D=-(pC-|p|^2A),\\
    pD-|p|^2B=0,
    \end{cases}
    \Longrightarrow
    \begin{cases}
    A=-\overline{p}B,\\
    D=\overline{p}B,\\
    C=-\dfrac{\overline{p}}{p}B.
    \end{cases}
\end{gather*}
Thus, this case gives
\begin{gather*}
    \psi(z)=\dfrac{|p|^2z-p}{\overline{p}z-|p|^2};
\end{gather*}
but this is impossible as $\psi$ is not a self-mapping of $\D$ (see Lemma \ref{202202242145}).
\end{proof}

\subsection{Basic properties of $W_{f,g}$}
The closed graph of the maximal operator $W_{f,g,\max}$ is left to the reader as its proof is similar to those used in \cite{zbMATH07216720}.

\begin{prop}
The maximal operator $W_{f,g,\max}$ is closed.
\end{prop}

Consequently, we get a criterion for the boundedness of the maximal operator $W_{f,g,\max}$.
\begin{cor}\label{cor-bdd}
The maximal operator $W_{f,g,\max}$ is bounded if and only if $\text{dom}(W_{f,g,\max})=\spc$.
\end{cor}

The following lemma will be used frequently to prove the main results as it shows that kernel functions always belong to the domain of $W_{f,g,\max}^*$.
\begin{lem}\label{lem-W*Kz}
Suppose that the operator $W_{f,g}$ is densely defined. Then equality $W_{f,g}^*K_z=\overline{f(z)}K_{g(z)}$ holds for every $z\in\dom$.
\end{lem}

We take a while to focus on the very restrictive category of operators induced by the following functions
\begin{gather}
    \varphi_\tha(z)
    =\left(\dfrac{\tha_1-z_1}{1-z_1\overline{\tha_1}},\dfrac{\tha_2-z_2}{1-z_2\overline{\tha_2}},
    \cdots,\dfrac{\tha_d-z_d}{1-z_d\overline{\tha_d}}\right)
\end{gather}
and
\begin{gather}\label{form-psi-tha}
    \psi_\tha(z)=K_{\tha}(z)/\|K_{\tha}\|.
\end{gather}

\begin{prop}\label{prop-202105032140}
Let $\tha\in\dom$. Let $\cala_{\tha,\max}$ be the maximal operator generated by $\widetilde{[\psi_\tha,\varphi_\tha]}$. Then the operator $\cala_{\tha,\max}$ is unitary on $\spc$.
\end{prop}
\begin{proof}
First, we show that $\text{dom}(\cala_{\tha,\max})=\spc$ and $\|\cala_{\tha,\max}h\|=\|h\|$ for every $h\in\spc$. Indeed, for $h\in\spc$, we consider
\begin{gather*}
I=\int\limits_{\dom}|\widetilde{[\psi_\tha,\varphi_\tha]}h(z)|^2\prod_{j=1}^d(1+\ell_j)(1-|z_j|^2)^{\ell_j}\,dA(z_j),
\end{gather*}
which implies, after doing the change of variables $x_j=\frac{\theta_j-z_j}{1-\overline{\theta_j}z_j}$, that
\begin{gather*}
    I
    =\int\limits_{\dom}\prod_{j=1}^d|h(x)|^2\prod_{j=1}^d(1+\ell_j)(1-|x_j|^2)^{\ell_j}\,dA(x_j),
\end{gather*}
as desired. Next, by Lemma \ref{lem-W*Kz}, we have
\begin{gather*}
    \cala_{\tha,\max}\cala_{\tha,\max}^*K_z(y)=\overline{\psi_\tha(z)}\psi_\tha(y)K_{\varphi_\tha(z)}(\varphi_\tha(y))
    =K_z(y),
\end{gather*}
which implies, as the linear span of kernel functions is dense, that $\cala_{\tha,\max}^*\cala_{\tha,\max}=I$.
\end{proof}

\subsection{Conjugations}
Let $\{U_1,U_2\}$ be a partition of $\zd$, that is
\begin{gather}\label{cond-UV}
    U_1\cup U_2=\zd,\quad U_1\cap U_2=\emptyset.
\end{gather}
Let $\bfp=(\bfp_1,\bfp_2,\cdots,\bfp_{|U_1|})\in\D^{|U_1|}\setminus\{0\}$ and $\bfq=(\bfq_1,\bfq_2.\cdots,\bfq_{|U_2|})\in\T^{|U_2|}$. Consider the operator
\begin{gather}
    \calc_{\bfp,\bfq}h(z)=\vartheta_{\bfp,U}(z)\overline{h(\overline{\ome(z)})},
\end{gather}
where $\vartheta_{\bfp,U}:\dom\to\C,\ome=(\ome_1,\ome_2,\cdots,\ome_d):\dom\to\dom$ are given by
\begin{gather}
    \vartheta_{\bfp,U}(z)
    =\prod_{j\in U_1}\dfrac{(1-|\bfp_j|^2)^{1+\ell_j/2}}{(1-\overline{\bfp_j}z_j)^{\ell_j+2}},\\
    \ome_j(z_j)=
    \begin{cases}
    \dfrac{\overline{\bfp_j}}{\bfp_j}\cdot\dfrac{\bfp_j-z_j}{1-\overline{\bfp_j}z_j}
    ,\quad\text{if $j\in U_1$},\\
    \\
    \bfq_j z_j,\quad\text{if $j\in U_2$}.
    \end{cases}
\end{gather}

\begin{lem}\label{lem-Cq*Kz}
The operator $\calc_{\bfp,\bfq}$ is a conjugation and moreover it satisfies $\calc_{\bfp,\bfq}K_z=\vartheta_{\bfp,U}(z)K_{\overline{\ome(z)}}$.
\end{lem}
\begin{proof}
For every $h\in\spc$, we consider the integral
\begin{gather*}
    I=\int\limits_{\dom}|\vartheta_{\bfp,U}(z)h(\overline{\ome(z)})|^2\prod_{j=1}^d(1+\ell_j)(1-|z_j|^2)^{\ell_j}\,dA(z_j),
\end{gather*}
which implies, after doing the change of variables $x=\ome(z)$, that $I=\|h\|^2$. The equality shows that the operator $\calc_{\bfp,\bfq}$ is isometric. On one hand, we have $\overline{\ome(\overline{\ome(z)})}=z$ and on the other hand, we observe $\vartheta_{\bfp,U}(z)\overline{\vartheta_{\bfp,U}(\overline{\ome(z)})}=1$. Thus, the operator $\calc_{\bfp,\bfq}$ is involutive.
\end{proof}

\section{Real symmetry}\label{sec3}
Recall that a linear operator $Q$ is called \emph{real symmetric} if the equality $Q=Q^*$ holds. In this section, we are concerned with how the function-theoretic properties of the symbols affect the real symmetry of the corresponding weighted composition operator, and vice versa. As a consequence, we show that a real symmetric weighted composition operator must be bounded.

We start the section by a lemma which focuses on symbol computation.
\begin{lem}\label{lem-her-20210505}
Let $d\in\Z_{\geq 1}$ and $\ell\in\Z_{\geq 0}^d$. Let $f:\dom\to\C,g=(g_1,g_2,\cdots,g_d):\dom\to\dom$ be analytic functions. Suppose that
\begin{gather}\label{cond-20210328}
    \overline{f(z)}K_{g(z)}(u)=f(u)K_z(g(u))\quad\forall z,u\in\dom.
\end{gather}
Then the following conclusions hold.
\begin{enumerate}
    \item The functions $f(\cdot),g(\cdot)$ are of forms
    \begin{gather}\label{form-f-her}
    f(z)=\bfc K_{\bfa}(z),\quad 
    g_\kappa(z)=\bfa_\kappa+\dfrac{\bfb_\kappa z_\kappa}{1-\overline{\bfa_\kappa}z_\kappa}\quad
    \forall\kappa\in\zd,
\end{gather}
where coefficients satisfy
\begin{gather}\label{her-cond-1}
    \bfc\in\R,\quad\bfb=(\bfb_1,\bfb_2,\cdots,\bfb_d)\in\R^d,\quad\bfa=(\bfa_1,\bfa_2,\cdots,\bfa_d)\in\dom
\end{gather}
and
\begin{gather}\label{her-cond-2}
    |\bfa_\kappa(\bfb_\kappa-|\bfa_\kappa|^2+1)|+|\bfb_\kappa|\leq 1-|\bfa_\kappa|^2\quad
    \forall\kappa\in\zd.
\end{gather}
\item If the functions are given in item (1), then $W_{f,g,\max}$ is bounded on $\spc$.
\end{enumerate}
\end{lem}
\begin{proof}
(1) Taking into account the form of kernel functions, \eqref{cond-20210328} can be expressed in the following
\begin{gather}\label{eq-202103252108}
    \overline{f(z)}\prod_{j=1}^d\left(1-g_j(u)\overline{z_j}\right)^{\ell_j+2}
    =f(u)\prod_{j=1}^d\left(1-u_j\overline{g_j(z)}\right)^{\ell_j+2}.
\end{gather}
Take arbitrarily $\kappa\in\zd$. Differentiating \eqref{eq-202103252108} with respect to the variable $u_\kappa$ (i.e. taking derivative $\pa_{u_\kappa}$), we obtain
\begin{gather}
    \nonumber-\overline{f(z)}\sum_{t=1}^d
    (\ell_t+2)\left(1-g_t(u)\overline{z_t}\right)^{\ell_t+1}\overline{z_t}\pa_{u_\kappa}(g_t(u))
    \prod_{j\ne t}\left(1-g_j(u)\overline{z_j}\right)^{\ell_j+2}\\
    \nonumber=\pa_{u_\kappa}(f(u))\prod_{j=1}^d\left(1-u_j\overline{g_j(z)}\right)^{\ell_j+2}\\
    \label{eq-202103280828}-f(u)(\ell_\kappa+2)\left(1-u_\kappa\overline{g_\kappa(z)}\right)^{\ell_\kappa+1}\overline{g_\kappa(z)}
    \prod_{j\ne\kappa}\left(1-u_j\overline{g_j(z)}\right)^{\ell_j+2}.
\end{gather}
Equation \eqref{eq-202103280828} divided by \eqref{eq-202103252108} is equal to
\begin{gather}\label{eq-202103271045}
    -\sum_{t=1}^d(\ell_t+2)\dfrac{\overline{z_t}\pa_{u_\kappa}(g_t(u))}{1-g_t(u)\overline{z_t}}
    =\dfrac{\pa_{u_\kappa}(f(u))}{f(u)}-(\ell_\kappa+2)\dfrac{\overline{g_\kappa(z)}}{1-u_\kappa\overline{g_\kappa(z)}}.
\end{gather}
In particular with $z=0$, we get
\begin{gather}\label{eq-202103281033}
    \dfrac{\pa_{u_\kappa}(f(u))}{f(u)}=(\ell_\kappa+2)\dfrac{\overline{g_\kappa(0)}}{1-u_\kappa\overline{g_\kappa(0)}}.
\end{gather}
Setting
\begin{gather*}
    H(u)=(u_\kappa\overline{g_\kappa(0)}-1)^{\ell_\kappa+2}f(u),
\end{gather*}
then
\begin{gather*}
    \pa_{u_\kappa}(f(u))=\dfrac{\pa_{u_\kappa}(H(u))}{(u_\kappa\overline{g_\kappa(0)}-1)^{\ell_\kappa+2}}
    -\dfrac{(\ell_\kappa+2)\overline{g_\kappa(0)}H(u)}{(u_\kappa\overline{g_\kappa(0)}-1)^{\ell_\kappa+3}}
\end{gather*}
and so \eqref{eq-202103281033} becomes $\pa_{u_\kappa}(H(u))=0$. For that reason, $(u_\kappa\overline{g_\kappa(0)}-1)^{\ell_\kappa+2}f(u)$ is a function of variables $u_1,u_2,\cdots,u_{\kappa-1},u_{\kappa+1},\cdots,u_d$. Since $\kappa$ is arbitrary, by Lemma \ref{20220224} the function $f(\cdot)$ must be of form in \eqref{form-f-her}. Next, we find the function $g(\cdot)$ as follows.

{\bf Claim:} $g_\kappa(\cdot)$ is a function of one variable $u_\kappa$.

Indeed, after taking derivative $\pa_{\overline{z_\kappa}}$ in \eqref{eq-202103271045}, the following is obtained
\begin{gather*}
    \dfrac{\pa_{u_\kappa}(g_\kappa(u))}{(1-g_\kappa(u)\overline{z_\kappa})^2}
    =\dfrac{\overline{\pa_{z_\kappa}(g_\kappa(z))}}{(1-u_\kappa\overline{g_\kappa(z)})^2}
\end{gather*}
or equivalently to saying
\begin{gather}\label{eq-202103260816}
    \pa_{u_\kappa}(g_\kappa(u))(1-u_\kappa\overline{g_\kappa(z)})^2
    =\overline{\pa_{z_\kappa}(g_\kappa(z))}(1-g_\kappa(u)\overline{z_\kappa})^2.
\end{gather}
For every $s\in\zd\setminus\{\kappa\}$, we differentiate \eqref{eq-202103260816} with respect to the variable $u_s$
\begin{gather}\label{eq-202103280839}
    \pa_{u_s}\circ\pa_{u_\kappa}(g_\kappa(u))(1-u_\kappa\overline{g_\kappa(z)})^2
    =-2\overline{\pa_{z_\kappa}(g_\kappa(z))}(1-g_\kappa(u)\overline{z_\kappa})
    \overline{z_\kappa}\pa_{u_s}(g_\kappa(u)).
\end{gather}
If there is $u_\star\in\dom$ with $\pa_{u_\kappa}(g_\kappa(u_\star))=0$, then \eqref{eq-202103260816} gives
\begin{gather*}
    \pa_{z_\kappa}(g_\kappa(z))=0\quad\forall z\in\dom,
\end{gather*}
and hence by \eqref{eq-202103280839} we have
\begin{gather*}
    \pa_{u_s}\circ\pa_{u_\kappa}(g_\kappa(u))=0\quad\forall u\in\dom;
\end{gather*}
meaning that $g_\kappa(\cdot)$ is a function of one variable $u_\kappa$. Now consider the case when $\pa_{u_\kappa}g_\kappa(\cdot)\not\equiv 0$. Equation \eqref{eq-202103280839} divided by \eqref{eq-202103260816} is equal to
\begin{gather*}
    \dfrac{\pa_{u_s}\circ\pa_{u_\kappa}(g_\kappa(u))}{\pa_{u_\kappa}(g_\kappa(u))}
    =-\dfrac{2\overline{z_\kappa}\pa_{u_s}(g_\kappa(u))}{1-g_\kappa(u)\overline{z_\kappa}},
\end{gather*}
which implies, by way of equating coefficients of $\overline{z_\kappa}$, that
\begin{gather*}
    \pa_{u_s}\circ\pa_{u_\kappa}(g_\kappa(u))=0=\pa_{u_s}(g_\kappa(u)).
\end{gather*}
The claim is proven. 

Suppose that $g_\kappa(u)=\psi(u_\kappa)$ for some function $\psi:\D\to\C$. Thus, \eqref{eq-202103260816} is simplified to \eqref{eq-202103262022}. By Lemma \ref{lem-202103262022}, the function $g_\kappa(\cdot)$ is of the desired form.

(2) Since the function $f$ is bounded, it is enough to show that the composition operator $C_{g,\max}$ is bounded. Before proving this, we denote
\begin{gather*}
    \Omega=\{j\in\zd:b_j=0\},\quad \Delta=\{j\in\zd:b_j\ne 0\}.
\end{gather*}
Suppose that $n_1,n_2,\cdots,n_t$ and $m_1,m_2,\cdots,m_s$ are elements of $\Omega,\Delta$, respectively; meaning
\begin{gather*}
    \Omega=\{n_1,n_2,\cdots,n_t\},\quad\Delta=\{m_1,m_2,\cdots,m_s\}.
\end{gather*}
For each $z=(z_1,z_2,\cdots,z_d)\in\C^d$, we express $z=(z_\Omega,z_\Delta)$, where $z_\Omega=(z_{n_1},z_{n_2},\cdots,z_{n_t})$ and $z_\Delta=(z_{m_1},z_{m_2},\cdots,z_{m_s})$. Then $g(z)=(\bfa_\Omega,g_\Delta(z_\Delta))$ and 
\begin{gather*}
    C_{g,\max}h(z)=h(\bfa_\Omega,g_\Delta(z_\Delta))=h_{\bfa_\Omega}(g_\Delta(z_\Delta))=C_{g_\Delta,\max}h_{\bfa_\Omega}(z_\Delta).
\end{gather*}
Note that the Jacobian determinant of $g_\Delta$ is
\begin{gather*}
    J=\prod_{j\in\Delta}\dfrac{|b_j|^2}{|1-\overline{a_j}z_j|^4}\ne 0,
\end{gather*}
so by \cite[Theorem 10]{zbMATH05285682}, the operator $C_{g_\Delta,\max}$ is bounded and together with Lemma \ref{lem-202106091438}, we get the boundedness of the operator $C_{g,\max}$.
\end{proof}

With all preparation in place we give one of the main results of this section. The following theorem provides a useful criteria to determine whether a maximal weighted composition operator is real symmetric or not.

\begin{thm}\label{thm-her-1}
Let $d\in\Z_{\geq 1}$ and $\ell\in\Z_{\geq 0}^d$. Let $f:\dom\to\C,g:\dom\to\dom$ be analytic functions. Then the operator $W_{f,g,\max}$ is real symmetric on $\spc$ if and only if the functions $f(\cdot),g(\cdot)$ are of forms in \eqref{form-f-her}, where coefficients verify \eqref{her-cond-1}-\eqref{her-cond-2}. In this case, the operator $W_{f,g,\max}$ is bounded.
\end{thm}
\begin{proof}
Suppose that the operator $W_{f,g,\max}$ is real symmetric on $\spc$, which gives $W_{f,g,\max}^*=W_{f,g,\max}$. In particular, the following is obtained
\begin{gather}\label{eq-202105041444}
    W_{f,g,\max}^*K_z=W_{f,g,\max}K_z\quad\forall z\in\dom.
\end{gather}
By Lemma \ref{lem-W*Kz}, \eqref{eq-202105041444} becomes \eqref{cond-20210328} and Lemma \ref{lem-her-20210505} gives the necessary condition.

For the sufficient condition, take $f(\cdot),g(\cdot)$ as in the statement of the theorem. Lemma \ref{lem-her-20210505} shows that the operator $W_{f,g,\max}$ is bounded. By \eqref{form-f-her}-\eqref{her-cond-2} and Lemma \ref{lem-W*Kz}, the operator verifies \eqref{eq-202105041444} and so it must be real symmetric.
\end{proof}

The following result shows that the real symmetry cannot be separated from the maximal domain; in other words, a real symmetric weighted composition operator must be maximal.
\begin{thm}\label{thm-her-2}
Let $d\in\Z_{\geq 1}$ and $\ell\in\Z_{\geq 0}^d$. Let $f:\dom\to\C,g:\dom\to\dom$ be analytic functions. Then the operator $W_{f,g}$ is real symmetric on $\spc$ if and only if it verifies two assertions.
\begin{enumerate}
\item \eqref{form-f-her}-\eqref{her-cond-2} hold. 
    \item The operator $W_{f,g}$ is maximal; that is $W_{f,g}=W_{f,g,\max}$.
\end{enumerate}
In this case, the operator $W_{f,g}$ is bounded.
\end{thm}
\begin{proof}
The implication $``\Longleftarrow"$ is proven in Theorem \ref{thm-her-1} and the remaing task is to prove the implication $``\Longrightarrow"$. Indeed, since $W_{f,g}\preceq W_{f,g,\max}$, by \cite[Proposition 1.6]{KS}, we have
$$
W_{f,g,\max}^*\preceq W_{f,g}^*=W_{f,g}\preceq W_{f,g,\max}.
$$
Lemma \ref{lem-W*Kz} shows that $K_z\in\text{dom}(W_{f,g,\max}^*)$, and so,
$$
W_{f,g,\max}^*K_z(u)=W_{f,g,\max}K_z(u)\quad\forall z,u\in\dom.
$$
By Lemmas \ref{lem-W*Kz} and \ref{lem-her-20210505}, conditions \eqref{form-f-her}-\eqref{her-cond-2} hold, and hence, by Theorem \ref{thm-her-1}, the operator $W_{f,g,\max}=W_{f,g,\max}^*$. Using this equality, item (2) is proven as follows
$$
W_{f,g}\preceq W_{f,g,\max}=W_{f,g,\max}^*\preceq W_{f,g}^*=W_{f,g}.
$$
\end{proof}

\section{Unitary property}\label{sec4}
Recall that a bounded linear operator $Q$ is called \emph{unitary} if the equality $QQ^*=Q^*Q=I$ holds. In this section, we describe all weighted composition
operators that are unitary. The following lemma gives a partial characterization of the operator $W_{f,g,\max}$ under the assumption that the symbol $g(\cdot)$ fixes $0$.

\begin{lem}\label{lem-univ}
Let $d\in\Z_{\geq 1}$ and $\ell\in\Z_{\geq 0}^d$. Let $f:\dom\to\C,g=(g_1,g_2,\cdots,g_d):\dom\to\dom$ be analytic functions. Suppose that
\begin{gather}\label{cond-202105021108}
    \overline{f(z)}f(u)K_{g(z)}(g(u))=K_z(u)\quad\forall z,u\in\dom.
\end{gather}
If $g(0)=0$, then
\begin{gather}
    f(\cdot)\equiv\bfc,\quad 
    g_\kappa(z)=\bfa_{\kappa}z_{\phi(\kappa)},
\end{gather}
where
\begin{gather}\label{cond-202105041603}
    \text{$\phi:\zd\to\zd$ is bijective and}\quad\bfc,\bfa_\kappa\in\T.
\end{gather}
\end{lem}
\begin{proof}
Setting $z=0$ in \eqref{cond-202105021108}, we find $\overline{f(0)}f(u)=1$, which means $f(\cdot)\equiv\bfc$, where $\bfc\in\T$. Consequently, taking into account the explicit form of kernel functions, \eqref{cond-202105021108} is reduced to the following
\begin{gather}\label{eq-202105021643}
    \prod_{j=1}^d(1-g_j(u)\overline{g_j(z)})^{\ell_j+2}
    =\prod_{j=1}^d(1-u_j\overline{z_j})^{\ell_j+2}.
\end{gather}
Let $\kappa\in\zd$. Taking derivative $\pa_{u_\kappa}$, \eqref{eq-202105021643} becomes
\begin{gather}
    \nonumber-\sum_{t=1}^d(\ell_t+2)(1-g_t(u)\overline{g_t(z)})^{\ell_t+1}\pa_{u_\kappa}(g_t(u))\overline{g_t(z)}
    \prod_{j\ne t}(1-g_j(u)\overline{g_j(z)})^{\ell_j+2}\\
    \label{eq-202105021650}=-(\ell_\kappa+2)\overline{z_\kappa}(1-u_\kappa\overline{z_\kappa})^{\ell_\kappa+1}
    \prod_{j\ne\kappa}(1-u_j\overline{z_j})^{\ell_j+2}.
\end{gather}
Equation \eqref{eq-202105021650} divided by \eqref{eq-202105021643} is equal to
\begin{gather*}
    \sum_{t=1}^d(\ell_t+2)\dfrac{\pa_{u_\kappa}(g_t(u))\overline{g_t(z)}}{1-g_t(u)\overline{g_t(z)}}
    =(\ell_\kappa+2)\dfrac{\overline{z_\kappa}}{1-u_\kappa\overline{z_\kappa}}.
\end{gather*}
In particular with $u=0$, we get
\begin{gather}\label{eq-202105022112}
    \sum_{t=1}^d(\ell_t+2)\pa_{u_\kappa}(g_t(0))\overline{g_t(z)}
    =(\ell_\kappa+2)\overline{z_\kappa}\quad\forall\kappa\in\zd.
\end{gather}
Let
\begin{gather*}
\mba=
\begin{pmatrix}
(\ell_1+2)\overline{\pa_{u_1}(g_1(0))} & (\ell_2+2)\overline{\pa_{u_1}(g_2(0))} &\cdots &(\ell_d+2)\overline{\pa_{u_1}(g_d(0))}\\
& & &\\
(\ell_1+2)\overline{\pa_{u_2}(g_1(0))} & (\ell_2+2)\overline{\pa_{u_2}(g_2(0))} &\cdots &(\ell_d+2)\overline{\pa_{u_2}(g_d(0))}\\
\vdots &\vdots & \ddots & \vdots\\
(\ell_1+2)\overline{\pa_{u_d}(g_1(0))} & (\ell_2+2)\overline{\pa_{u_d}(g_2(0))} &\cdots &(\ell_d+2)\overline{\pa_{u_d}(g_d(0))}
\end{pmatrix}
\end{gather*}
and $\mbb$ be the diagonal matrix given by
\begin{gather*}
    \mbb=\text{diag}\left(\ell_1+2,\ell_2+2,\cdots,\ell_d+2\right).
\end{gather*}
Now equation \eqref{eq-202105022112} is rewritten as $\mba g(z)=\mbb z$, which gives $g(z)=\mbh z$, where $\mbh=(h_{i,j})_{d\times d}=\mba^{-1}\mbb$.

Fix $p\in\zd$. Setting $z=(0,\ldots,0,z_p,0,\ldots,0)$ in \eqref{eq-202105021643}, where $z_p$ is the $p$-th coordinate, we obtain
\begin{gather}\label{eq-202105030934}
    \prod_{j=1}^d\left(1-g_j(u)\overline{h_{j,p}z_p}\right)^{\ell_j+2}
    =(1-u_p\overline{z_p})^{\ell_p+2}.
\end{gather}
We continue with choosing $u=(0,\ldots,0,u_p,0,\ldots,0)$, where $u_p$ is the $p$-th coordinate, and then \eqref{eq-202105030934} is reduced to the following
\begin{gather*}
    \prod_{j=1}^d(1-|h_{j,p}|^2u_p\overline{z_p})^{\ell_j+2}=(1-u_p\overline{z_p})^{\ell_p+2}.
\end{gather*}
Hence, since $u_p,z_p\in\D$ are arbitrary, coefficients $h_{j,p}$ verify
\begin{gather}\label{eq-202105030930}
    \prod_{j=1}^d(1-|h_{j,p}|^2x)^{\ell_j+2}=(1-x)^{\ell_p+2}\quad\forall x\in\D.
\end{gather}
The highest power of the left-hand side is $\sum\limits_{j=1}^d(\ell_j+2)$; meanwhile of the right-hand side is $\ell_p+2$. Thus, the coefficient of $x^{\sum\limits_{j=1}^d(\ell_j+2)}$ in \eqref{eq-202105030930} must be $0$; meaning that
\begin{gather*}
    \prod_{j=1}^d|h_{j,p}|^2=0.
\end{gather*}
Denote
\begin{gather*}
    V_p=\{j\in\zd:h_{j,p}=0\},\quad U_p=\{j\in\zd:h_{j,p}\ne 0\}.
\end{gather*}
If $U_p=\emptyset$, then $h_{j,p}=0$ for every $j\in\zd$; but this is impossible as the matrix $\mbh$ is invertible. Now consider the situation when $U_p\ne\emptyset$. Setting $u=w_{-p}:=(w_1,\ldots,w_{p-1},0,w_{p+1},\ldots,w_d)$ in \eqref{eq-202105030934}, where $0$ is the $p$-th coordinate, we get
\begin{gather*}
    \prod_{j=1}^d\left(1-g_j(w_{-p})\overline{h_{j,p}z_p}\right)^{\ell_j+2}=1.
\end{gather*}
Through equating coefficients of $\overline{z_p}$ and then using the fact that $h_{j,p}\ne 0$ for $j\in U_p$, we have
\begin{gather}
    \nonumber g_j(w_{-p})=0\Longrightarrow h_{j,t}=0\quad\forall j\in U_p,t\in\zd\setminus\{p\},
\end{gather}
and so
\begin{gather}
    \label{eq-202105031340}g_j(u_1,u_2,\cdots,u_d)=h_{j,p}u_p\quad\forall j\in U_p.
\end{gather}
Note that equation \eqref{eq-202105030930} can be rewritten in the following form
\begin{gather*}
    \prod_{j\in U_p}(1-|h_{j,p}|^2x)^{\ell_j+2}=(1-x)^{\ell_p+2},
\end{gather*}
which implies, after equating coefficients of $x^{\ell_p+2}$, that
\begin{gather*}
    \prod_{j\in U_p}|h_{j,p}|^{\ell_j+2}=1.
\end{gather*}
Since $g(z)=\mbh z$ is a self-mapping of $\dom$, we have
\begin{gather*}
    |h_{j,p}|\leq 1\quad\forall j\in U_p,p\in\zd\quad
    \Longrightarrow\quad\prod_{j\in U_p}|h_{j,p}|^{\ell_j+2}\leq 1\quad\forall p\in\zd.
\end{gather*}
Thus,
\begin{gather*}
    |h_{j,p}|= 1\quad\forall j\in U_p,p\in\zd.
\end{gather*}

{\bf Claim:} For every $p\in\zd$, $U_p$ is a singleton set.

Since $U_p,p\in\zd$ are subsets of $\zd$, it is enough to show that the family $\{U_p:p\in\zd\}$ consists of disjoint sets. Indeed, assume in contrary that there exist $p,q\in\zd$ with $p\ne q$ such that $U_p\cap U_q\ne\emptyset$. Let $t\in U_p\cap U_q$. It follows from \eqref{eq-202105031340} that
\begin{gather*}
    g_t(u)=h_{t,p}u_p=h_{t,q}u_q;
\end{gather*}
but this is impossible as $|h_{t,p}|= 1=|h_{t,q}|$ and $p\ne q$.

Let us define the map $\eta:\zd\to\zd$ by setting $\eta(p)=j$, where $j,p\in\zd$ satisfy \eqref{eq-202105031340}. Then $\eta$ is bijective and \eqref{cond-202105041603} holds, where $\phi=\eta^{-1}$.
\end{proof}

With the help of Lemma \ref{lem-univ}, we give a complete characterization of unitary weighted composition operators. 
\begin{thm}
Let $d\in\Z_{\geq 1}$ and $\ell\in\Z_{\geq 0}^d$. Let $f:\dom\to\C,g:\dom\to\dom$ be analytic functions. Then the operator $W_{f,g,\max}$ is unitary on $\spc$ if and only if
\begin{gather}\label{uni-form-fg}
f(z)=\bfc\prod_{j=1}^d\dfrac{(1-|\theta_j|^2)^{1+\ell_j/2}}{(1-\bfa_j\overline{\theta_j}z_{\phi(j)})^{\ell_j+2}},\quad 
g_\kappa(z)=\dfrac{\bfa_\kappa(\overline{\bfa_\kappa}\tha_\kappa-z_{\phi(\kappa)})}
{1-\bfa_\kappa\overline{\tha_\kappa}z_{\phi(\kappa)}},
\end{gather}
where coefficients satisfy \eqref{cond-202105041603} and
\begin{gather}
    \theta=(\theta_1,\theta_2,\cdots,\theta_d)\in\D^d.
\end{gather}
\end{thm}
\begin{proof}
Note that the sufficient condition can be obtained by arguments similar to those used in Proposition \ref{prop-202105032140}. We prove the necessary condition as follows. Suppose that the operator $W_{f,g,\max}$ is unitary on $\spc$. Put $\theta=g(0)$ and define
\begin{gather*}
    F=f\cdot\psi_\tha\circ g,\quad G=\varphi_\tha\circ g.
\end{gather*}
A direct computation gives $\widetilde{[F,G]}=\widetilde{[f,g]}\widetilde{[\psi_\tha,\varphi_\tha]}$ and so we get
\begin{gather*}
    W_{F,G,\max}=W_{f,g,\max}\cala_{\tha,\max}.
\end{gather*}
Using Proposition \ref{prop-202105032140}, we can show that the operator $W_{F,G,\max}$ is unitary with $G(0)=\varphi_\tha(g(0))=0$. We have
\begin{gather*}
    K_z=W_{F,G,\max}W_{F,G,\max}^*K_z=W_{F,G,\max}\left(\overline{F(z)}K_{G(z)}\right)\\
    =\overline{F(z)} F\cdot K_{G(z)}\circ G\\
    \Longrightarrow\quad K_z(u)=\overline{F(z)}F(u)K_{G(z)}(G(u))\quad\forall z,u\in\dom.
\end{gather*}
The line above allows us to use Lemma \ref{lem-univ} and the proof is complete.
\end{proof}

\section{Complex symmetry}\label{sec5}
The section studies which the symbols give rise to weighted composition operators that are complex symmetric with respect to the conjugation $\calc_{\bfp,\bfq}$. Such operators are called \emph{$\calc_{\bfp,\bfq}$-symmetric}. Like as the real symmetry, a $\calc_{\bfp,\bfq}$-symmetric weighted composition operator must be bounded. As a byproduct, we obtain the interesting fact that real symmetric weighted composition operators are $\calc_{\bfp,\bfq}$-symmetric corresponding an adapted and highly relevant selection of the parameters $\bfp,\bfq$.

To study the necessary condition for a weighted composition operator to be $\calc_{\bfp,\bfq}$-symmetric, we apply the symmetric condition to kernel functions. It turns out that the symbols generating such operators can be precisely computed.
\begin{lem}\label{lem-202105210827}
Let $d\in\Z_{\geq 1},\ell\in\Z_{\geq 0}^d$. Let $U_1,U_2$ with condition \eqref{cond-UV} and $\bfp\in\D^{|U_1|}\setminus\{0\},\bfq\in\T^{|U_2|}$. Let $\bfr\in\dom$, where
\begin{gather*}
    \bfr_t=
    \begin{cases}
    \bfp_t\quad\text{if $t\in U_1$},\\
    0\quad\text{if $t\in U_2$}.
    \end{cases}
\end{gather*}
Let $f:\dom\to\C,g=(g_1,g_2,\cdots,g_d):\dom\to\dom$ be analytic functions. Suppose that
\begin{gather}\label{eq-202105261537}
    \overline{\vartheta_{\bfp,U}(z)}\vartheta_{\bfp,U}(g(\overline{\omega(z)}))f(\overline{\omega(z)})
    K_{\overline{\omega(g(\overline{\omega(z)}))}}(u)
    =f(u)K_z(g(u))\quad\forall z,u\in\D^d.
\end{gather}
Then the following assertions hold.
\begin{enumerate}
    \item The functions are of forms
    \begin{gather}\label{eq-202105261617}
    f(u)=\widetilde{c}K_{\overline{\ome(g(\bfr))}}(u),\quad 
    g_\kappa(u_\kappa)=
    \begin{cases}
    G_\kappa+\dfrac{E_\kappa}{u_\kappa+F_\kappa}
    \quad\text{if $\kappa\in U_1$}\\
    \\
    \al_\kappa+\dfrac{\beta_\kappa\bfq_\kappa u_\kappa}{1-\al_\kappa\bfq_\kappa u_\kappa}\quad\text{if $\kappa\in U_2$},
    \end{cases}
\end{gather}
where coefficients satisfy
\begin{gather}
\label{202208121603}G_\kappa\in\D\quad\text{if $\kappa\in U_1,E_\kappa=0$},\\
    \label{cond-202106030907}p-G_\kappa|p_\kappa|^2=-F_\kappa|p_\kappa|^2+(G_\kappa F_\kappa+E_\kappa)\overline{p_\kappa}\quad\text{if $\kappa\in U_1,E_\kappa\ne 0$},\\
    \label{202208121558}|(G_\kappa F_\kappa+E_\kappa)\overline{F_\kappa}-G_\kappa|+|E_\kappa|\leq|F_\kappa|^2-1\quad\text{if $\kappa\in U_1,E_\kappa\ne 0$},\\
    \label{202208121628}\alpha_\kappa\in\D\quad\text{if $\kappa\in U_2,\beta_\kappa=0$},\\
\label{cond-202105210836}
    |\al_\kappa+\overline{\al_\kappa}(\beta_\kappa-\al_\kappa^2)|+|\beta_\kappa|\leq 1-|\alpha_\kappa|^2\quad\text{if $\kappa\in U_2,\beta_\kappa\ne 0$}.
\end{gather}
\item If the functions are given in item (1), then $W_{f,g,\max}$ is bounded on $\spc$.
\end{enumerate}
\end{lem}
\begin{proof}
(1) Setting $z=\overline{\omega(y)}$ in \eqref{eq-202105261537}, we find
\begin{gather*}
    \overline{\vartheta_{\bfp,U}(\overline{\ome(y)})}\vartheta_{\bfp,U}(g(y))f(y)K_{\overline{\ome(g(y))}}(u)
    =f(u)K_{\overline{\ome(y)}}(g(u))\quad\forall y,u\in\dom.
\end{gather*}
Consequently, taking into account the explicit form of kernel functions, we get
\begin{gather}
    \nonumber f(y)
    \prod_{j\in U_1}(1-|\bfp_j|^2)^{\ell_j+2}\prod_{j=1}^d[1-g_j(u)\ome_j(y_j)]^{\ell_j+2}\\
    \label{eq-202105261608}=f(u)\prod_{j\in U_1}(1-\overline{\bfp_j}g_j(y))^{\ell_j+2}(1-\bfp_j\ome_j(y_j))^{\ell_j+2}
    \prod_{j=1}^d[1-u_j\ome_j(g_j(y))]^{\ell_j+2}.
\end{gather}
Let $\kappa\in\zd$. After differentiating with respect to the variable $u_\kappa$, the equation above becomes
\begin{gather}
    \nonumber -f(y)\prod_{j\in U_1}(1-|\bfp_j|^2)^{\ell_j+2}
    \sum_{t=1}^d(\ell_t+2)[1-g_t(u)\ome_t(y_t)]^{\ell_t+1}\ome_t(y_t)\pa_{u_\kappa}(g_t(u))\\
    \nonumber\times\prod_{j\ne t}[1-g_j(u)\ome_j(y_j)]^{\ell_j+2}\\
    \nonumber=\pa_{u_\kappa}(f(u))\prod_{j\in U_1}(1-\overline{\bfp_j}g_j(y))^{\ell_j+2}(1-\bfp_j\ome_j(y_j))^{\ell_j+2}
    \prod_{j=1}^d[1-u_j\ome_j(g_j(y))]^{\ell_j+2}\\
    \nonumber-f(u)\prod_{j\in U_1}(1-\overline{\bfp_j}g_j(y))^{\ell_j+2}(1-\bfp_j\ome_j(y_j))^{\ell_j+2}\\
    \label{eq-202105261609}\times(\ell_\kappa+2)[1-u_\kappa\ome_\kappa(g_\kappa(y))]^{\ell_\kappa+1}
    \ome_\kappa(g_\kappa(y))
    \prod_{j\ne\kappa}[1-u_j\ome_j(g_j(y))]^{\ell_j+2}.
\end{gather}
Equation \eqref{eq-202105261609} divided by \eqref{eq-202105261608} is equal to
\begin{gather}\label{eq-202105261625}
    -\sum_{t=1}^d(\ell_t+2)\dfrac{\ome_t(y_t)\pa_{u_\kappa}(g_t(u))}{1-g_t(u)\ome_t(y_t)}
    =\dfrac{\pa_{u_\kappa}(f(u))}{f(u)}
    -\dfrac{(\ell_\kappa+2)\ome_\kappa(g_\kappa(y))}{1-u_\kappa\ome_\kappa(g_\kappa(y))}.
\end{gather}
Setting $y=\bfr$, we observe
\begin{gather*}
    \dfrac{\pa_{u_\kappa}(f(u))}{f(u)}
    =\dfrac{(\ell_\kappa+2)\ome_\kappa(g_\kappa(\bfr))}{1-u_\kappa\ome_\kappa(g_\kappa(\bfr))}.
\end{gather*}
Using the arguments similar to \eqref{eq-202103281033}, the function $f(\cdot)$ is of form as in \eqref{eq-202105261617}. Setting $\xi=\ome_\kappa\circ g_\kappa$, \eqref{eq-202105261625} is rewritten as
\begin{gather}\label{eq-202106022212}
    -\sum_{t=1}^d(\ell_t+2)\dfrac{\ome_t(y_t)\pa_{u_\kappa}(g_t(u))}{1-g_t(u)\ome_t(y_t)}
    =\dfrac{\pa_{u_\kappa}(f(u))}{f(u)}
    -\dfrac{(\ell_\kappa+2)\xi(y)}{1-u_\kappa\xi(y)}.
\end{gather}
We differentiate \eqref{eq-202106022212} with respect to the variable $y_{\kappa}$, where $j\in\{1,2\}$, and then
\begin{gather}
\label{eq-202105261827}
    \dfrac{\ome_\kappa^{\odot}(y_\kappa)\pa_{u_\kappa}(g_\kappa(u))}{(1-\ome_\kappa(y_\kappa)g_\kappa(u))^2}
    =\dfrac{\pa_{y_{\kappa}}(\xi(y))}{(1-u_\kappa\xi(y))^2}
\end{gather}
or equivalently to saying that
\begin{gather}
    \label{eq-202105261815}\ome_\kappa^{\odot}(y_\kappa)\pa_{u_\kappa}(g_\kappa(u))(1-u_\kappa\xi(y))^2
    =\pa_{y_{\kappa}}(\xi(y))(1-\ome_\kappa(y_\kappa)g_\kappa(u))^2,
\end{gather}
where we denote
\begin{gather*}
    \ome_\kappa^{\odot}(y_\kappa)=
    \begin{cases}
    \dfrac{\overline{\bfp_{\kappa}}}{\bfp_\kappa}\cdot\dfrac{|\bfp_\kappa|^2-1}
    {(1-\overline{\bfp_\kappa}z_{\kappa})^2}\quad\text{if $\kappa\in U_1$},\\
    \bfq_\kappa\quad\text{if $\kappa\in U_2$}.
    \end{cases}
\end{gather*}

{\bf Claim 1:} $g_\kappa(\cdot)$ is a function of one variable $u_\kappa$.

Let $s\in\zd\setminus\{\kappa\}$. Taking derivative $\pa_{u_s}$ on both sides of \eqref{eq-202105261815} gives
\begin{gather}
    \nonumber\ome_\kappa^{\odot}(y_\kappa)\pa_{u_s}\circ\pa_{u_\kappa}(g_\kappa(u))(1-u_\kappa\xi(y))^2\\
    \label{eq-202105261816}
    =-\pa_{y_{\kappa}}(\xi(y))2(1-\ome_\kappa(y_\kappa)g_\kappa(u))\ome_\kappa(y_\kappa)\pa_{u_s}(g_\kappa(u)).
\end{gather}
If there is $u_\star\in\dom$ for which $\pa_{u_\kappa}(g_\kappa(u_\star))=0$, then \eqref{eq-202105261815} gives
\begin{gather*}
    \pa_{y_{\kappa}}(\xi(y))=0\quad\forall y\in\dom,
\end{gather*}
and hence by \eqref{eq-202105261816},
\begin{gather*}
    \pa_{u_s}\circ\pa_{u_\kappa}(g_\kappa(u))=0\quad\forall u\in\dom;
\end{gather*}
meaning that $g_\kappa(\cdot)$ is a function of one variable $u_\kappa$. Now consider the situation when $\pa_{u_\kappa}g_\kappa(\cdot)\not\equiv 0$. Equation \eqref{eq-202105261816} divided by \eqref{eq-202105261815} is equal to the following
\begin{gather*}
    \dfrac{\pa_{u_s}\circ\pa_{u_\kappa}(g_\kappa(u))}{\pa_{u_\kappa}(g_\kappa(u))}
    =-\dfrac{2\ome_\kappa(y_\kappa)\pa_{u_s}(g_\kappa(u))}{1-\ome_\kappa(y_\kappa)g_\kappa(u)},
\end{gather*}
which implies, after equating coefficients of $\ome_\kappa(y_\kappa)$, that
\begin{gather*}
    \pa_{u_s}\circ\pa_{u_\kappa}(g_\kappa(u))=0=\pa_{u_s}(g_\kappa(u));
\end{gather*}
meaning that $g_\kappa(\cdot)$ is a function of one variable $u_\kappa$.

{\bf Claim 2:} $\xi(\cdot)$ is a function of one variable $y_{\kappa}$, too.

Let $t\in\zd\setminus\{\kappa\}$. We proceed into taking derivative $\pa_{y_t}$ on both sides of \eqref{eq-202105261827}
\begin{gather*}
    0=\pa_{y_t}\left(\dfrac{\ome_\kappa^{\odot}(y_\kappa)\pa_{u_\kappa}(g_\kappa(u))}{(1-\ome_\kappa(y_\kappa)g_\kappa(u))^2}\right)
    =\pa_{y_t}\left(\dfrac{\pa_{y_{\kappa}}(\xi(y))}{(1-u_\kappa\xi(y))^2}\right)\\
    \Longrightarrow
    0=\pa_{y_t}\circ\pa_{y_{\kappa}}(\xi(y))(1-u_\kappa\xi(y))
    +2u_\kappa\pa_{y_{\kappa}}(\xi(y))\pa_{y_t}(\xi(y)).
\end{gather*}
After equating coefficients of $u_\kappa$, we get
\begin{gather*}
    \pa_{y_t}\circ\pa_{y_{\kappa}}(\xi(y))=0=\pa_{y_{\kappa}}(\xi(y))\cdot\pa_{y_t}(\xi(y));
\end{gather*}
meaning that $\xi(\cdot)$ is a function of one variable $y_{\kappa}$. Thus, we make use of Lemma \ref{lem-202105202050} (when $\kappa\in U_2$) and Lemma \ref{lem-202105291538} (when $\kappa\in U_1$).

(2) The part is the same as those of Lemma \ref{lem-her-20210505}(2) but we give a proof, for a completeness of exposition. Since the function $f(\cdot)$ is bounded, it is enough to show that the composition operator $C_{g,\max}$ is bounded. Before proving this, we fix some symbols used. Denote
\begin{gather*}
    \Omega=\{j\in\zd:g_j\equiv\text{const}\},\quad\Delta=\left\{j\in U_1:g_j\not\equiv\text{const}\right\},\quad\Xi=\left\{j\in U_2:g_j\not\equiv\text{const}\right\}.
\end{gather*}
Suppose that $n_1,n_2,\cdots,n_t$ and $m_1,m_2,\cdots,m_s$ and $c_1,c_2,\cdots,c_r$ are elements of $\Omega,\Delta,\Xi$, respectively; meaning
\begin{gather*}
    \Omega=\{n_1,n_2,\cdots,n_t\},\quad\Delta=\{m_1,m_2,\cdots,m_s\},\quad\Xi=\{c_1,c_2,\cdots,c_r\}.
\end{gather*}
For each $z=(z_1,z_2,\cdots,z_d)\in\C^d$, we express $z=(z_\Omega,z_\Delta,z_\Xi)$, where $z_\Omega=(z_{n_1},z_{n_2},\cdots,z_{n_t})$ and $z_\Delta=(z_{m_1},z_{m_2},\cdots,z_{m_s})$ and $z_\Xi=(z_{c_1},z_{c_2},\cdots,z_{c_r})$. Then $g(z)=(\bfa_\Omega,g_\Delta(z_\Delta),g_\Xi(z_\Xi))$ and 
\begin{gather*}
    C_{g,\max}h(z)=h(\bfa_\Omega,g_\Delta(z_\Delta),g_\Xi(z_\Xi))=h_{\bfa_\Omega}(g_\Delta(z_\Delta),g_\Xi(z_\Xi))=C_{(g_\Delta,g_\Xi),\max}h_{\bfa_\Omega}(z_\Delta,z_\Xi).
\end{gather*}
Note that the Jacobian determinant of $(g_\Delta,g_\Xi)$ is
\begin{gather*}
    J=\prod_{j\in\Delta}\left|\dfrac{E_j}{(z_j+F_j)^2}\right|^2\cdot\prod_{\kappa\in\Xi}\left|\dfrac{\beta_\kappa}{(1-\alpha_\kappa\bfq_\kappa z_\kappa)^2}\right|^2\ne 0,
\end{gather*}
so by \cite[Theorem 10]{zbMATH05285682}, the operator $C_{g_\Delta,\max}$ is bounded and together with Lemma \ref{lem-202106091438}, we get the boundedness of the operator $C_{g,\max}$.
\end{proof}

Lemma \ref{lem-202105210827} provides a necessary condition for a weighted composition operator to be $\calc_{\bfp,\bfq}$-symmetric. It turns out that the assertion in Lemma \ref{lem-202105210827} is also a sufficient condition.
\begin{thm}\label{thm-cs-1}
Let $d\in\Z_{\geq 1},\ell\in\Z_{\geq 0}^d$. Let $U_1,U_2$ with condition \eqref{cond-UV} and $\bfp\in\D^{|U_1|}\setminus\{0\},\bfq\in\T^{|U_2|}$. Suppose that $f:\dom\to\C,g=(g_1,g_2,\cdots,g_d):\dom\to\dom$ are analytic functions. Then the operator $W_{f,g,\max}$ is $\calc_{\bfp,\bfq}$-symmetric on $\spc$ if and only if the functions $f(\cdot),g(\cdot)$ are of forms in \eqref{eq-202105261617}, where coefficients verify \eqref{202208121603}-\eqref{cond-202105210836}. In this case, the operator $W_{f,g,\max}$ is bounded.
\end{thm}
\begin{proof}
Suppose that the operator $W_{f,g,\max}$ is $\calc_{\bfp,\bfq}$-symmetric on $\spc$, which gives
\begin{gather*}
    \calc_{\bfp,\bfq}W_{f,g,\max}^*\calc_{\bfp,\bfq}=W_{f,g,\max}.
\end{gather*}
In particular, the following is obtained
\begin{gather}\label{eq-202105201542}
    \calc_{\bfp,\bfq}W_{f,g,\max}^*\calc_{\bfp,\bfq}K_z=W_{f,g,\max}K_z\quad\forall z\in\dom.
\end{gather}
By Lemma \ref{lem-W*Kz}, \eqref{eq-202105201542} becomes \eqref{eq-202105261537} and so we can make use of Lemma \ref{lem-202105210827} to get the necessary condition.

For the sufficient condition, take $f(\cdot),g(\cdot)$ as in the statement of the theorem. Lemma \ref{lem-202105210827} shows that the operator $W_{f,g,\max}$ is bounded. By \eqref{eq-202105261617}-\eqref{cond-202105210836} and Lemma \ref{lem-W*Kz}, the operator verifies \eqref{eq-202105201542} and so it must be $\calc_{\bfp,\bfq}$-symmetric.
\end{proof}

Theorem \ref{thm-cs-1} characterizes maximal weighted composition operators that are $\calc_{\bfp,\bfq}$-symmetric. The following theorem proves that the maximal domain and boundedness are consequences of the $\calc_{\bfp,\bfq}$-symmetry.
\begin{thm}\label{thm-cs-2}
Let $d\in\Z_{\geq 1},\ell\in\Z_{\geq 0}^d$. Let $U_1,U_2$ with condition \eqref{cond-UV} and $\bfp\in\D^{|U_1|}\setminus\{0\},\bfq\in\T^{|U_2|}$. Suppose that $f:\dom\to\C,g=(g_1,g_2,\cdots,g_d):\dom\to\dom$ are analytic functions. Then the operator $W_{f,g}$ is $\calc_{\bfp,\bfq}$-symmetric on $\spc$ if and only if it verifies two assertions.
\begin{enumerate}
\item \eqref{202208121603}-\eqref{cond-202105210836} hold. 
    \item The operator $W_{f,g}$ is maximal; that is $W_{f,g}=W_{f,g,\max}$.
\end{enumerate}
In this case, the operator $W_{f,g}$ is bounded.
\end{thm}
\begin{proof}
The implication $``\Longleftarrow"$ is proven in Theorem \ref{thm-cs-1} and the remaining task is to prove the implication $``\Longrightarrow"$. Indeed, since $W_{f,g}\preceq W_{f,g,\max}$, by \cite[Proposition 1.6]{KS}, we have
$$
\calc_{\bfp,\bfq}W_{f,g,\max}^*\calc_{\bfp,\bfq}\preceq\calc_{\bfp,\bfq}W_{f,g}^*\calc_{\bfp,\bfq}=W_{f,g}\preceq W_{f,g,\max}.
$$
Lemma \ref{lem-W*Kz} shows that $K_z\in\text{dom}(W_{f,g,\max}^*)$, and so,
$$
\calc_{\bfp,\bfq}W_{f,g,\max}^*\calc_{\bfp,\bfq}K_z(u)=W_{f,g,\max}K_z(u)\quad\forall z,u\in\dom.
$$
By Lemmas \ref{lem-W*Kz} and \ref{lem-202105210827}, conditions \eqref{eq-202105261617}-\eqref{cond-202105210836} hold, and hence, by Theorem \ref{thm-cs-1}, the operator $W_{f,g,\max}$ is $\calc_{\bfp,\bfq}$-symmetric. Using this, item (2) is proven as follows
$$
W_{f,g}\preceq W_{f,g,\max}=\calc_{\bfp,\bfq}W_{f,g,\max}^*\calc_{\bfp,\bfq}\preceq \calc_{\bfp,\bfq}W_{f,g}^*\calc_{\bfp,\bfq}=W_{f,g}.
$$
\end{proof}

\begin{cor}
Let $d\in\Z_{\geq 1}$ and $\ell\in\Z_{\geq 0}^d$. Let $f:\dom\to\C,g=(g_1,g_2,\cdots,g_d):\dom\to\dom$ be analytic functions. If the operator $W_{f,g}$ is real symmetric, then it is complex symmetric.
\end{cor}
\begin{proof}
Suppose that the operator $W_{f,g}$ is real symmetric. By Theorem \ref{thm-her-2}, the functions $f(\cdot),g(\cdot)$ satisfy \eqref{form-f-her}-\eqref{her-cond-2}. It then is $\calc_{\bfp,\bfq}$-symmetric, where $U_2=\zd$ and
\begin{gather*}
    \alpha_\kappa=
    \begin{cases}
    \bfa_\kappa\quad\text{if $\bfa_\kappa\ne 0$},\\
    0\quad\text{if $\bfa_\kappa=0$},
    \end{cases}
    \bfq_\kappa=
    \begin{cases}
    \dfrac{\overline{\bfa_\kappa}}{\bfa_\kappa}\quad\text{if $\bfa_\kappa\ne 0$},\\
    1\quad\text{if $\bfa_\kappa=0$},
    \end{cases}
    \beta_\kappa=
    \begin{cases}
    \dfrac{\bfa_\kappa\bfb_\kappa}{\overline{\bfa_\kappa}}\quad\text{if $\bfa_\kappa\ne 0$},\\
    \bfb_\kappa\quad\text{if $\bfa_\kappa=0$}.
    \end{cases}
\end{gather*}
\end{proof}

\section*{Acknowledgements}
The paper was completed during a scientific stay of P.V. Hai at the Vietnam Institute for Advanced Study in Mathematics (VIASM). He would like to thank the VIASM for financial support and hospitality.

\nocite{*}
\bibliographystyle{plain}
\bibliography{refs}
\end{document}